\documentclass{amsart}
\usepackage{amsmath,amssymb,amsthm,tikz,hyperref,mathtools}
\usepackage[abbrev]{amsrefs}

\newcommand{\abs}[1]{\left\vert {#1}\right\vert}
\newcommand{\norm}[1]{\left\Vert {#1}\right\Vert}
\newcommand{\avg}[1]{\left\langle {#1}\right\rangle}
\newcommand{\ep}{\epsilon}

\newtheorem{theorem}{Theorem}[section]
\newtheorem{lemma}[theorem]{Lemma}
\newtheorem{corollary}[theorem]{Corollary}
\newtheorem{proposition}[theorem]{Proposition}
\newtheorem{problem}[theorem]{Problem}
\newtheorem{example}[theorem]{Example}
\newtheorem{remark}[theorem]{Remark}
\theoremstyle{definition}
\newtheorem{definition}[theorem]{Definition}

\numberwithin{equation}{section}

\newcommand{\Piv}{\Pi^v}
\newcommand{\Mv}{M^v}

\DeclareMathOperator{\BHO}{BHO}
\title{Weak-type estimates for the Bergman projection on planar domains}
\author[A. W. Green]{A. Walton Green}
\address{A. Walton Green, Department of Mathematics, Illinois State University, Normal, IL 61761}
\email{awgree1@ilstu.edu}
\author[C. B. Stockdale]{Cody B. Stockdale}
\address{Cody B. Stockdale, School of Mathematical and Statistical Sciences, Clemson University, Clemson, SC 29634, USA}
\email{cbstock@clemson.edu}
\author[B. Sweeting]{Brandon Sweeting}
\address{Brandon Sweeting, Department of Mathematics, Washington University in St. Louis, St. Louis, MO 63130, USA}
\email{sweeting@wustl.edu}
\author[N. A. Wagner]{Nathan A. Wagner}
\thanks{N. A. W. is supported by the National Science Foundation, DMS Grant No. 2549719}
\address{Nathan A. Wagner, Department of Mathematical Sciences, George Mason University, Fairfax, VA 22030, USA}
\email{nwagner8@gmu.edu}
\begin{document}

\keywords{Bergman projection, weak-type estimates, Bekoll\'e-Bonami, mixed-weighted weak-type estimates}

\subjclass[2020]{30H25, 42B20}

\begin{abstract}
    We investigate the relationship between the weak-type regularity of the Bergman projection, $\Pi_{\Omega}$, of a simply connected domain $\Omega \subset \mathbb{C}$ and the boundary geometry of $\Omega$ in terms of a conformal map $\psi\colon\mathbb{D}\rightarrow\Omega$. We show that $\Pi_{\Omega}$ is of weak-type $(1,1)$ whenever $|\psi'|$ is in the Bekoll\'e-Bonami class $B_1$, give a more general necessary condition for the weak-type $(p,p)$ bounds of $\Pi_{\Omega}$ when $1\leq p<\infty$, and establish sharpened sufficient conditions for the weak-type bounds when $p>1$. Our results follow from a reformulation in terms of mixed-weighted weak-type inequalities for $\Pi_{\mathbb{D}}$. We provide several applications.
\end{abstract}

\maketitle

\section{Introduction}
Given a domain $\Omega \subset \mathbb C$, the Bergman space $A^2(\Omega)$ is the closed subspace of $L^2(\Omega)$ consisting of analytic functions. As a reproducing kernel Hilbert space, each $A^2(\Omega)$ admits a Bergman kernel $K_\Omega \colon \Omega \times \Omega \to \mathbb C$ such that integration against this kernel forms the Bergman projection, namely the orthogonal projection $\Pi_\Omega \colon L^2(\Omega) \to A^2(\Omega)$. More precisely, for $f \in L^2(\Omega)$ and $z \in \Omega$, it holds that 
\begin{equation*}
\Pi_{\Omega}f(z)= \int_{\Omega} K_{\Omega}(z,w) f(w) \, dA(w), \label{BergmanL2Rep}
\end{equation*}  
where $A$ denotes the Lebesgue area measure on $\mathbb C$.
In the special case of the unit disk, we have the formula 
$$
   K(z,w) := K_{\mathbb{D}}(z,w) = \frac{1}{\pi (1-z\overline{w})^2}. 
$$
In this paper, we consider simply connected domains $\Omega$, in which case $\Pi_\Omega$ and $\Pi:=\Pi_{\mathbb D}$ can be related through a conformal map $\psi\colon\mathbb D \to \Omega$. Using this connection, we investigate the mapping properties of $\Pi_\Omega$ in terms of the regularity of $\psi$. 
This program was initiated by 
D. Bekoll\'e and A. Bonami in \cites{Bek82,Bek86,BB78} to study the boundedness of $\Pi_\Omega$ on $L^p(\Omega)$ for $1<p<\infty$. 
They observed that such an estimate is equivalent to a  weighted $L^p$ bound for the Bergman projection on the unit disk, which can be characterized using an adaptation of Calder\'on-Zygmund theory for Muckenhoupt weights. More precisely, with $v=|\psi'|$, $\Pi_\Omega$ is bounded on $L^p(\Omega)$ if and only if $v^{2-p}$ satisfies the $B_p$ condition
        \begin{equation}\label{e:Bp-intro}  \sup_Q \frac{\norm{v^{-1}}_{L^p(Q,v^2)} \norm{v^{-1}}_{L^{p'}(Q,v^2)} }{|Q|} < \infty,\end{equation}
where $\frac{1}{p}+\frac{1}{p'}=1$ and the supremum is taken over all Carleson boxes $Q$,
which are subsets of $\mathbb{D}$ associated with a point $e^{2 \pi i\theta} \in \mathbb T$ and length $\ell \in (0,1]$ having the form
    \[ Q=Q(\theta,\ell) = \left\{ re^{2 \pi i\phi} \colon 1-\ell \le r <1 , \, |\theta-\phi| \le \tfrac{\ell}{2} \right\}.\]
The $B_p$ property is a variant of the classical Muckenhoupt $A_p$ condition, the latter of which characterizes the weighted $L^p(\mathbb{R})$-boundedness of the Hilbert transform and Hardy-Littlewood maximal function; see \cites{CMP2011,Duo2001,GR1985,Gra14} for more on the real-variable theory of $A_p$ weights.

Our work starts by extending this line of research to the endpoint case $p=1$.
It is commonly understood that $\Pi_\Omega$  
is typically unbounded from $L^1(\Omega)$ to itself; see \cite{Dall22} for precise formulations of this fact. However, regularity may be salvaged when the target space is replaced with the larger Lorentz space $L^{1,\infty}(\Omega)$. 
Generally, for $1 \le p < \infty$, we say that an operator $T$ is of weak-type $(p,p)$ if it maps $L^p(\Omega)$ to $L^{p,\infty}(\Omega)$ boundedly, that is 
$$
    \|Tf\|_{L^{p,\infty}(\Omega)} \coloneq \sup_{\lambda>0}\lambda |\{z \in \Omega\colon |Tf(z)|>\lambda\}|^{1/p} \lesssim \|f\|_{L^p(\Omega)}
$$
for all $f \in L^p(\Omega)$. Note that if $T$ is bounded on $L^p(\Omega)$, then $T$ is of weak-type $(p,p)$ by Chebyshev's inequality; for this reason, it is natural to refer to the $L^p(\Omega)$ bound as the strong-type $(p,p)$ property of $T$. See \cites{Bek82, Bek86, DHZ2001, McN94, SW23, CNSW25} for more on the weak-type regularity of the Bergman projection. 

Since the condition $|\psi'|^{2-p} \in B_p$ from \eqref{e:Bp-intro} characterizes the strong-type $(p,p)$ bounds of $\Pi_\Omega$, it is also sufficient for $\Pi_\Omega$ to be of weak-type $(p,p)$ for $1<p<\infty$. Taking the limit of this condition 
as $p \searrow 1$, 
it is natural to conjecture that $\Pi_\Omega$ is of weak-type $(1,1)$ if $v=|\psi'|$ belongs to the class $B_1$, which means 
    \begin{equation*}\label{e:B1} 
    \sup_Q\frac{\norm{v}_{L^1(Q)} \norm{v^{-1}}_{L^\infty(Q)}}{|Q|} < \infty.\end{equation*}
However, this has remained elusive until now, mainly because the change of variables in $L^{p,\infty}$ is more delicate than in a typical $L^p$ space -- the weight $v=|\psi'|$ appears both as a multiplier on the operator, as well as in the underlying measure on $\mathbb D$. Such weighted inequalities are referred to as mixed-weighted weak-type due to the mixed contribution of the weight. 
In particular, treating the multiplier portion of the weight $v$ causes additional complications absent in the strong-type case, where the weight may be absorbed entirely by the measure. We provide more discussion on mixed-weighted weak-type estimates in Subsection \ref{Subsection:MixedWeighted}. \looseness=-1

Using this perspective, Bekoll\'e showed that $|\psi'|^2 \in B_1$ is sufficient for $\Pi_{\Omega}$ to be of weak-type $(1,1)$ in \cite{Bek86}. Our first main result weakens this hypothesis to $|\psi'| \in B_1$. 
\begin{theorem}\label{thm:B1}
Let $\Omega \subset \mathbb C$ be a simply connected domain
and $\psi\colon \mathbb D \to \Omega$ be a conformal map. If $|\psi'| \in B_1$, then $\Pi_\Omega$ is of weak-type $(1,1)$. 
\end{theorem}

Boundedly extending $\Pi_{\Omega}$ between function spaces is useful for determining functional analytic properties of Bergman spaces and relevant to planar PDEs and geometric function theory; see \cite{Hed02}. We apply Theorem \ref{thm:B1} to develop properties of the Bergman-type spaces 
$A^1(\Omega)$ and $A^{1,\infty}(\Omega)$, 
which respectively are the closed subspaces of $L^1(\Omega)$ and $L^{1,\infty}(\Omega)$ 
consisting of analytic functions. Recall that by construction, $\Pi_\Omega$ reproduces functions in $A^2(\Omega)$ and orthogonally projects $L^2(\Omega)$ onto $A^2(\Omega)$. Our main corollary suitably extends these properties to $A^1(\Omega)$ and $A^{1,\infty}(\Omega)$. \looseness=-1 
\begin{corollary}\label{c:A1}
Let $\Omega \subset \mathbb C$ be a simply connected domain and $\psi\colon \mathbb D \to \Omega$ be a conformal map. If $|\psi'| \in B_1$, then
 	\begin{itemize}
 	\item[I.] $\Pi_\Omega f = f$ for all $f \in A^1(\Omega)$, and 
 	\item[II.] $\Pi_\Omega \colon L^1(\Omega)\to A^{1,\infty}(\Omega)$.
 	\end{itemize}
\end{corollary}
\noindent We prove Corollary \ref{c:A1} and other related applications in Section \ref{s:RC}.

\subsection{Necessary conditions}\label{Subsection:Intro-Necessity}
The novelty of Theorem \ref{thm:B1} notwithstanding, $|\psi'| \in B_1$ is still not necessary for $\Pi_\Omega$ to be of weak-type $(1,1)$. Indeed, for $\psi(z) = \frac{1}{z-1}$ and $\Omega = \psi(\mathbb{D})$, one has that $\Pi_{\Omega}$ is of weak-type $(1,1)$ (since $\Omega$ is a half-space), but $|\psi'|\not\in B_1$. Instead, for any $1 \le p < \infty$, we have the following necessary condition on $v:=|\psi'|$ for $\Pi_{\Omega}$ to be of weak-type $(p,p)$: 
    \begin{equation}\label{e:Bp-nec-intro} \sup_{Q} \frac{\norm{v^{-1}}_{L^{p,\infty}(Q,v^2)} \norm{v^{-1}}_{L^{p'}(Q,v^2)}}{|Q|} < \infty,
    \end{equation}
where 
$\|v^{-1}\|_{L^{p'}(Q,v^2)}=\|v^{-1}\|_{L^{\infty}(Q)}$ when $p=1$. 
\begin{theorem}\label{thm:Nec}
    Let $\Omega \subset \mathbb{C}$ be a simply connected domain, $\psi \colon \mathbb{D} \rightarrow \Omega$ be a conformal map, and $1 \leq p < \infty$. If $\Pi_{\Omega}$ is of weak-type $(p,p)$, then $|\psi'|$ satisfies \eqref{e:Bp-nec-intro}. 
\end{theorem}

Notice that if the $L^{p,\infty}(Q,v^2)$ norm in \eqref{e:Bp-nec-intro} was replaced by the larger $L^p(Q,v^2)$ norm, \eqref{e:Bp-nec-intro} would recover the condition \eqref{e:Bp-intro}, that $v^{2-p} \in B_p$, which characterizes the strong-type $(p,p)$ bounds of $\Pi_{\Omega}$.
However, \eqref{e:Bp-nec-intro} is strictly weaker;
for example, the weight $v\coloneq |\psi'|$ for $\psi(z) = \frac{1}{(z-1)^\alpha}$ and $1\le \alpha \le 2$ satisfies \eqref{e:Bp-nec-intro} for every $\frac{2\alpha}{\alpha-1} > p \ge \frac{2\alpha}{\alpha+1}$, but $v^{2-p}$ does not belong to $B_p$ when $p= \frac{2\alpha}{\alpha+1}$. 
It is possible that $v=|\psi'|$ satisfying \eqref{e:Bp-nec-intro} is also sufficient for the weak-type bounds of $\Pi_{\Omega}$, although this currently seems out of reach. Such a conjecture is connected to the problem of characterizing the multiplier weak-type weights for singular integrals, which is also unresolved; see \cites{MW77,Saw85,Swe25} and Section \ref{Subsection:MixedWeighted} below for further discussion. 
Nonetheless, in the case $p>1$, we provide a sufficient condition for the weak-type $(p,p)$ bounds in terms of the Bergman maximal function \looseness=-1
	\[ M f = \sup_{Q} 
    \frac{\norm{f}_{L^1(Q)}}{|Q|} \mathbf 1_Q \]
satisfying a dual weighted estimate.
For $1 \leq p< \infty$ and a measure space $(X,\mu)$, the Lorentz space $L^{p,1}(X,\mu)$ consists of all $f$ on $X$ such that 
$$
    \|f\|_{L^{p,1}(X,\mu)} := p\int_0^{\infty} \mu(\{x \in X \colon |f(x)| > \lambda\})^\frac{1}{p} d\lambda < \infty.
$$
We are now ready to state our final main result.
\begin{theorem}\label{thm:Bp}
 Let $\Omega \subset \mathbb{C}$ be a simply connected domain, $\psi \colon \mathbb{D} \rightarrow \mathbb{C}$ be a conformal map, $v:= |\psi'|$, and $1< p< \infty$. If $M^v$ is bounded from $L^{p',1}(\mathbb{D}, v^2)$ to $L^{p'}(\mathbb{D},v^2)$, then $\Pi_{\Omega}$ is of weak-type $(p,p)$.
\end{theorem}

Combining the strong-type results of Bekoll\'e and Bonami with our partial results in Theorems \ref{thm:B1}-\ref{thm:Bp} above, we can state a summary theorem to compare the various conditions and describe the state of the art.
\begin{theorem}\label{thm:summ}
    Let $\Omega \subset \mathbb{C}$ be a simply connected domain, $\psi \colon \mathbb{D} \rightarrow \mathbb{C}$ be a conformal map, $v:= |\psi'|$, and $1\leq p< \infty$. The conditions
    \begin{itemize}
        \item[A.] $v$ satisfies \eqref{e:Bp-nec-intro}, 
        \item[B.] $\Pi_\Omega$ is of weak-type $(p,p)$,
        \item[C.] $M^v \colon L^{p'}(\mathbb D,v^2) \to L^{p',1}(\mathbb D,v^2)$, and
        \item[D.] $v$ satisfies \eqref{e:Bp-intro}, i.e., $v^{2-p} \in B_p$
    \end{itemize}
    are related as follows: 
    \begin{itemize}
 	\item[I.] if $p>1$, then $D. \implies C. \implies B. \implies A.$, and 
 	\item[II.] if $p=1$, then $D. \implies B. \implies A.$\, .
 	\end{itemize}
\end{theorem}
In both cases, the implication of B. to A. is our Theorem \ref{thm:Nec} above. In the case $p>1$, the implication C. to B. is our Theorem \ref{thm:Bp} and D. to C. follows from the results of Bekoll\'e and Bonami for the strong-type $(p,p)$ bound. Finally, the remaining implication, D. to B. when $p=1$ is precisely our Theorem \ref{thm:B1}. 

Let us make three further comments on our results stated from this perspective. First, we can modify statement C. in such a way that the chain of implications is the same for all $p \ge 1$. We postpone this technical discussion until after the proof of Theorem \ref{thm:Pi-M1} in Remark \ref{r:unified}.

Second, to characterize the $\Omega$ such that $\Pi_\Omega$ is weak-type $(p,p)$, one should compare conditions A. and C. of Theorem \ref{thm:summ}. Is C. strictly stronger than A.? 

Third, we may also include in Theorem \ref{thm:summ} a slightly stronger boundedness result for the positive Bergman projection $\Pi_\Omega^+$ defined by
    \[ \Pi_\Omega^+ f (z) = \int_\Omega \abs{K_\Omega(z,w)} f(w) \, dw.\]
Immediately prior to statement B. in the chain on implications, we may place the statement
\begin{itemize}
    \item[B'.] $\Pi_\Omega^+$ is of weak-type $(p,p)$.
\end{itemize}
In fact, throughout the remainder of this paper, the upper estimates we prove will actually be for the a priori larger positive Bergman projection. However, at this time, the authors are unaware of any boundedness results which hold for $\Pi_\Omega$ yet fail for $\Pi_\Omega^+$.

\subsection{Mixed-weighted weak-type estimates}\label{Subsection:MixedWeighted}
As discussed above, our strategy is to reformulate the weak-type bounds of $\Pi_{\Omega}$ using the change of variables  
	\begin{equation}\label{e:transformation} C_\psi \circ \Pi_{\Omega} = \Pi \circ C_\psi, \end{equation}
where $C_\psi f := (f \circ \psi ) \psi'$. Note that \eqref{e:transformation} follows from the fact that $C_\psi$ is a unitary mapping from $L^2(\Omega)$ to $L^2(\mathbb{D})$ that preserves analytic functions. 
Given a function $f$ on $\Omega$, define $g =(f \circ \psi)\frac{\psi'}{|\psi'|}$. Then, applying \eqref{e:transformation}, we see that 
\[
    \left\{ z \in \Omega \colon \abs{\Pi_\Omega f(z)} > \lambda \right\} = \left\{ \psi(w) \colon \abs{ \Pi^vg(w)} > \lambda \right\}
\]
for any $\lambda>0$ and
\[
    \norm{g}_{L^p(\mathbb D,v^2)} = \norm{f}_{L^p(\Omega)},
\]
where $v = |\psi'|$ and $\Pi^v:= v^{-1}\Pi(v\,\cdot)$. We have established the following connection.\looseness=-1
\begin{proposition}\label{p:change}
    Let $\Omega \subset \mathbb{C}$ be a simply connected domain, $\psi\colon\mathbb{D}\rightarrow \Omega$ be a conformal map, $v = |\psi'|$, and $1\le p < \infty$. Then $\Pi_\Omega$ is weak-type $(p,p)$ if and only if the mixed-weighted weak-type estimate
\[ 
\|\Pi^vf\|_{L^{p,\infty}(\mathbb{D},v^2)} := \sup_{\lambda>0} \lambda v^2(\{ z \in \mathbb{D} \colon |\Pi^vf(z)| > \lambda \})^{\frac 1p} \lesssim \norm{f}_{L^p(\mathbb D, v^{2})}\]
holds for all $f \in L^p(\mathbb D,v^2)$. 

Furthermore, $\Pi_\Omega^+$ is of weak-type $(p,p)$ if and only if 
    \[ \norm{(\Pi^+)^v f}_{L^{p,\infty}(\mathbb D,v^2)} \lesssim \norm{f}_{L^p(\mathbb D,v^2)}\]
for all $f \in L^p(\mathbb D,v^2)$.
\end{proposition}
\begin{proof}
As stated above, the first statement follows directly from the two displays prior to the proposition. To establish the result for $\Pi_\Omega^+$, note that the transformation law \eqref{e:transformation} implies the kernel transformation formula
    \[ K_\Omega(z,w) = K_{\mathbb D}(\varphi(z),\varphi(w)) \varphi'(z) \overline{\varphi'(w)}, \quad \varphi = \psi^{-1}, \quad z,w \in \Omega,\]
from which the identity $C_\psi^+ \circ \Pi_\Omega^+ = \Pi^+ \circ C_\psi^+$ where $C_\psi^+f = (f \circ \psi) |\psi'|$. From this, one follows the same outline as for $\Pi_\Omega$.
\end{proof}

\noindent Our question can hence be addressed by studying the following general problem.
\begin{problem}\label{prob:MWWT}
Given $1 \leq p < \infty$, characterize the weights $u$ such that 
$\Pi^u$ is bounded from $L^p(\mathbb{D},u^2)$ to $L^{p,\infty}(\mathbb{D},u^2)$. 
\end{problem}

The bound of $\Pi^u$ from $L^p(\mathbb{D},u^2)$ to $L^{p,\infty}(\mathbb{D},u^2)$ is an instance of a type of bound first investigated by Sawyer in \cite{Saw85} for the Hardy-Littlewood maximal operator, and later expanded upon for Calder\'on-Zygmund singular integral operators in \cites{CUMP05, LOP19}. In particular, Sawyer's conjecture was proved in \cite{LOP19}, that if $T$ is a Calder\'on-Zygmund operator, $u \in A_1$, and $v \in A_{\infty}$, then $T^v := v^{-1}T(v\,\cdot)$ satisfies 
$$
    \|T^vf\|_{L^{1,\infty}(\mathbb{R}^n,uv)} \lesssim \|f\|_{L^1(\mathbb{R}^n,uv)}
$$
for all $f \in L^1(\mathbb{R}^n,uv)$. See \cites{CRR2020, OP2016, OPR2016, LOBP2019} for more on such mixed-weighted weak-type bounds for singular integrals with respect to Muckenhoupt weights. 

Similar mixed-weighted weak-type $(1,1)$ bounds were recently established in the Bergman setting in \cite{CNSW25}. Those results, specialized to the case $u=v$, require $v \in B_1 \cap B_p(v)$ for some $1 \le p < \infty$. Note that this condition is strictly stronger than $v \in B_1$ since it forces $v^2$ to be integrable -- in fact, $v^2$, the underlying measure in Problem \ref{prob:MWWT}, would be doubling. Indeed, the conformal weight $v(z)=|1-z|^{-1}$ arising from $\psi(z)=\log(1-z)$ belongs to $B_1$, but $v^2$ is not integrable up to the boundary of $\mathbb D$. 

To establish Theorem \ref{thm:B1}, we prove a different special case of \cite{CNSW25}*{Conjecture 1.6}, namely, the case $u=v \in B_1 \cap \text{BHO}$, where BHO is the class of bounded hyperbolic oscillation given in Definition \ref{d:BHO} below. 
\begin{theorem}\label{thm:sawyer}
    If $u \in B_1 \cap \text{BHO}$, then $\Pi^u$ is bounded from $L^1(\mathbb{D},u^2)$ to $L^{1,\infty}(\mathbb{D},u^2)$. 
\end{theorem}

\subsection{Organization}

In Section \ref{s:Preliminaries}, we collect preliminary results on weights and dyadic models. In Section \ref{s:Necessity}, we establish the necessity of \eqref{e:Bp-nec-intro} by proving Theorem \ref{thm:Nec}. In Section \ref{s:Reduction}, we reduce the problem to understanding weak-type bounds for the Bergman maximal operator and prove Theorem \ref{thm:Bp}. In Section \ref{s:max}, we establish the maximal function estimates and deduce Theorem \ref{thm:B1}. Finally, in Section \ref{s:RC}, we provide several applications of our results.

\section{Preliminaries}\label{s:Preliminaries}

We write $A\lesssim B$ if $A\leq CB$ for some $C>0$, and write $A\sim B$ if $A\lesssim  B\lesssim A$. 
\subsection{Weights} 
A weight is a nonnegative locally integrable function on $\mathbb D$. For $1 \le p < \infty$, we say a weight $u$ belongs to the class $B_p$ if there exists $C>0$ such that \looseness=-1
 	\begin{equation}\label{e:Bp} \avg{u}_Q \avg{u^{-1}}_{\frac{1}{p-1},Q} \le C \end{equation}
 for all Carleson boxes $Q$. Throughout, we use the local average notation for $E \subset \mathbb D$, $0 < r < \infty$, and a weight $u$:
 	\[ \avg{f}_{u,r,E} = \left(\frac{1}{u(E)} \int_E |f|^r u \, dA \right)^{\frac 1r}, \quad \avg{f}_{u,\infty,E} = \norm{f}_{L^\infty(E)}. \]
When $r=1$ or $u=1$, the parameters are dropped from the notation. For $1 \le p < \infty$, we denote by $[u]_{B_p}$ the $B_p$ characteristic of $u$, which is the infimum over all $C$ such that \eqref{e:Bp} holds. Since $\avg{f}_{\frac{1}{p-1},E}$ is a decreasing function of $p$, we have that $B_p \subset B_q$ with $[u]_{B_q} \le [u]_{B_p}$ whenever $p \leq q$.

There are many $B_\infty$-type conditions describing the limiting class as $p \rightarrow \infty$ that are, in general, not equivalent. However, these conditions all coincide when the weight is the Jacobian of a conformal map, as in our setting; see \cite{APR19}. We say a weight $u$ belongs to the class $B_\infty$ if there exist $C,\delta>0$ such that \looseness=-1
\begin{equation}\label{e:Binfty} 
    \frac{|E|}{|Q|} \le C \left( \frac{u(E)}{u(Q)} \right)^{\delta}
\end{equation}
for every Carleson box $Q$ and every measurable $E \subset Q$. By H\"older's inequality, we have that if $u\in B_p$ for $1\le p < \infty$, then $u$ satisfies the above $B_{\infty}$ condition with $C = [u]_{B_p}^{\frac{1}{p}}$ and $\delta= \frac{1}{p}$.

\subsection{Dyadic model}
\label{ss:dyad}
Given a Carleson box $Q=Q(\theta,\ell)$, let $T_Q$ denote the corresponding ``top half''
\[ T_Q=\left \{re^{2\pi\mathrm{i} \phi} \in \mathbb{D}\colon \frac{\ell}{2} \leq 1-r \leq \ell,\, |\theta-\phi| \leq \frac{\ell}{2} \right\}.\]
It is well-known that Carleson boxes can be well-approximated by those coming from two universal dyadic systems; say $\mathcal{D}_1, \mathcal{D}_2.$ That is, given an arbitrary box $Q$ there exist boxes $P,P' \in \mathcal{D}_1 \cup \mathcal{D}_2$ satisfying $P' \subset Q \subset P$ and $|Q| \sim |P| \sim |P'|.$ For example, one can take 
\begin{align*}
&\mathcal{D}_1= \left\{Q\left(\tfrac{j}{2^k},\tfrac{1}{2^k}\right) \colon j,k \in \mathbb{N}, \, 0 \leq j< 2^k \right\},\\
&\mathcal{D}_2= \left\{Q\left( \tfrac{j}{2^k} + \tfrac{2}{3}, \tfrac{1}{2^k}\right) \colon j,k \in \mathbb{N}, \,0 \leq j< 2^k\right\}.\end{align*}
The precise form of $\mathcal D_j$ is not important, but rather we note a few essential properties of any dyadic grid $\mathcal D$:
\begin{itemize}
	\item[(i)] $\{T_Q \colon Q \in \mathcal D\}$ forms a partition of $\mathbb D$, and
	\item[(ii)] any two $Q,P \in \mathcal D$ are either disjoint, or one is contained in the other.
\end{itemize}

For a Carleson box $Q$, we denote by $A_Q$ the averaging operator given by 
	\[ A_Q f = \avg{f}_Q \mathbf 1_Q.\]
It is by now well-known that $\Pi$ admits a ``sparse domination'' by a superposition of $A_Q$. In fact, this bound holds for the \textit{a priori} larger positive Bergman projection 
    \[ \Pi^+ f(z) = \int_{\mathbb D} \frac{f(w)}{|1-z\overline{w}|^2} \, dA(w).\] 

More precisely, for all $f \in L^1_{\text{loc}}(\Omega)$ and $z \in \Omega$, one has 
    \begin{equation}\label{e:sparse} \abs{\Pi f(z) } \le \Pi^+ |f|(z) \le \sum_{Q \in \mathcal D_1 \cup \mathcal D_2} A_Q|f|(z);\end{equation}
 see \cites{PR13, RTW17}. 
Therefore, to prove norm bounds involving $\Pi$, it suffices to replace $\Pi$ by $\sum_{Q \in \mathcal D} A_Q$ for some dyadic grid $\mathcal D$. This will be our approach, and hence all our results hold with $\Pi$ replaced by $\Pi^+$. Similarly, by well-approximating any Carleson box $Q$ by one from $\mathcal D_1 \cup \mathcal D_2$, it suffices to replace $M$ by the smaller dyadic maximal function
    \[ M^{\mathcal D}f = \sup_{Q \in \mathcal D} \avg{|f|}_Q \mathbf 1_Q,\] 
since $M \le M^{\mathcal D_1} + M^{\mathcal D_2}$ pointwise.

\section{\texorpdfstring{\boldmath Necessity of condition \eqref{e:Bp-nec-intro}}{}}\label{s:Necessity}

The goal of this section is to prove Theorem \ref{thm:Nec}, i.e., the necessity of \eqref{e:Bp-nec-intro} for the weak-type $(p,p)$ property of $\Pi_{\Omega}$. We recall that, for a general weight $u$ and $1 \leq p<\infty$, this property states that \begin{equation} \sup_Q \frac{\norm{ u^{-1} }_{L^{p,\infty}(Q,u^2)} \norm{u^{-1}}_{L^{p'}(Q,u^2)}}{|Q|} < \infty, \label{e:Bp-nec-sym} \end{equation} where $\|u^{-1}\|_{L^{p'}(Q,u^2)}$ is understood as $\|u^{-1}\|_{L^{\infty}(Q)}$ when $p=1$. We first show that \eqref{e:Bp-nec-sym} is necessary for the mixed-weighted bounds of the Bergman maximal operator. 
\begin{proposition}\label{prop:nec} 
If $u$ is a weight such that either
\[
    M^u\colon L^p(\mathbb{D},u^2) \to L^{p,\infty}(\mathbb{D},u^2) \quad\quad\text{for}\quad\quad1 \le p < \infty
\]
or
\[
    M^u \colon L^{p',1}(\mathbb{D},u^2) \to L^{p'}(\mathbb{D},u^2) \quad\quad\text{for}\quad\quad 1 < p < \infty,
\]
then $u$ satisfies \eqref{e:Bp-nec-sym}.
\end{proposition}
This result will follow from the lower bound $M \ge A_Q$ and the following straightforward characterization of the operator norms of single-scale averaging operators. Below, for two sets $E,F \subset \mathbb D$, we let $A_{E,F} := \avg{f}_E \mathbf 1_F$.
\begin{lemma}\label{l:single-scale}
 If $u$ is a weight and $1 \le p < \infty$, then 
    \[ \norm{ A_{E,F}^u \colon L^{p}(\mathbb{D},u^2) \to L^{p,\infty}(\mathbb{D},u^2) } = \frac{\norm{ u^{-1} }_{L^{p,\infty}(F,u^2)} \norm{u^{-1}}_{L^{p'}(E,u^2)}}{|E|}\]
    for any $E,F \subset \mathbb{D}$.
\end{lemma}
\begin{proof}[Proof of Lemma \ref{l:single-scale}]
    Apply H\"older's inequality with the underlying measure $u^2$ to
        \[ \int_E fu \, dA = \int_E f u^{-1} u^2 \, dA \le \norm{f}_{L^{p}(E,u^2)} \norm{u^{-1}}_{L^{p'}(E,u^2)}.\]
    But H\"older's inequality is sharp for a single scale. Indeed, fix $f= \mathbf 1_E u^{-p'/p}$ when $p>1$. And when $p=1$, choose a sequence of $L^1$-normalized $g_n$ such that $\int u^{-1}g_n \, dA \to \sup_E u^{-1}$, and set $f_n = g_n u^{-2}$.
\end{proof}

\begin{proof}[Proof of Proposition \ref{prop:nec}]

Lemma \ref{l:single-scale} gives
    \[ \norm{ A_Q^u \colon L^{p}(\mathbb{D},u^2) \to L^{p,\infty}(\mathbb{D},u^2) } = \frac{\norm{ u^{-1} }_{L^{p,\infty}(Q,u^2)}\norm{u^{-1}}_{L^{p'}(Q,u^2)}}{|Q|}. \] 
    The first statement follows since the Bergman maximal operator $M$ dominates each $A_Q$ pointwise. For the second statement, we repeat the same argument with the additional observation that $A_Q^u$ is self-adjoint on $L^2(\mathbb D,u^2)$, and hence
        \[ \norm{ A_Q^u \colon L^{p}(\mathbb{D},u^2) \to L^{p,\infty}(\mathbb{D},u^2) } = \norm{ A_Q^u \colon L^{p',1}(\mathbb{D},u^2) \to L^{p'}(\mathbb{D},u^2) }.\]
    The result follows.
\end{proof}

 For $k\geq 1$ and a Carleson box $Q=Q(\theta,\ell)$  we define $T_Q^k$ by to be an interior box-like region above $Q$ and sufficiently far away, i.e., 
$$ 
    T_Q^k := \left\{re^{i\phi} \colon \tfrac{k}{2}\ell \leq 1-r < \tfrac{k+1}{2}\ell, \,\, |\theta-\phi| \le \tfrac{\ell}{2}\right\}.
$$
Note that $T_Q^1$ equals $T_Q$, the ``top-half" of $Q$. 
We establish the following off-diagonal
pointwise lower bound relating $\Pi$ and $A_Q$. 

\begin{lemma}\label{l:avg-below}
If $k \ge 20$, then
\[k^2 |\Pi (\mathbf 1_Qf)| \mathbf 1_{T_Q^k} \gtrsim \avg{f}_Q \mathbf 1_{T_Q^k}\]
and 
\[ k^2 |\Pi (\mathbf 1_{T_Q^k}f)|\mathbf  1_{Q} \gtrsim \avg{f}_{T_Q^k} \mathbf 1_Q\]
for any Carleson box $Q$ with length $\ell(Q) \le \frac{1}{k}$ and any nonnegative function $f$ on $\mathbb{D}$. \looseness=-1
\end{lemma}
\begin{proof}
We follow the ideas of \cites{Bek82,CF74}, which simply rely on the smoothness of the kernel $K(z,w) := K_{\mathbb{D}}(z,w) = \frac{1}{(1-z\overline{w})^2}$. The main observation is that if $z \in T_Q^k$ and $w \in Q$, then by the triangle inequality 
	\[ \frac{k+3}{2} \ell(Q) \ge |1- z \overline{w}| \ge 
	\frac{k-2}{2} \ell(Q).\] 
If furthermore $w_0 \in Q$, then a routine computation shows that 
	\[ \left\vert \frac{1}{(1- z \overline{w})^2} - \frac{1}{(1- z \overline{w}_0)^2} \right\vert \le
	\frac{16}{(k-2)^3 \ell(Q)^2}.\]
Therefore, for $z \in T_Q^k$, we have that 
	\[ \abs{\Pi (\mathbf 1_Q f)(z)} \ge \frac{1}{|1- z \overline{w}_0|^2} \int_Q f \,dA - \frac{16}{(k-2)^3 \ell(Q)^2} \int_Q f\,dA, \]
which is bounded below by $\avg{f}_Q$ for $k$ large enough; namely if $\frac{16}{(k-2)^3} \le \frac{2}{(k+3)^2}$. The second inequality follows by symmetry. 
\end{proof}

To prove that \eqref{e:Bp-nec-sym} is necessary for the mixed-weighted weak-type bound of $\Pi$, 
we must use another special property of our weights.  
\begin{definition}\label{d:BHO}
We say a weight $u$ has bounded hyperbolic oscillation (BHO) if it is approximately constant on the top halves of Carleson boxes, that is, 
there exists $C>0$ such that
 	\begin{equation}\label{e:hyp} \sup_{z \in T_Q} u(z) \le C \inf_{w \in T_Q} u(w)\end{equation}
for all Carleson boxes $Q$. Let us denote by $[u]_{\BHO}$ the infimum over all $C$ such that \eqref{e:hyp} holds.
\end{definition}
\noindent It is well-known that for any conformal map $\psi \colon \mathbb D \to \mathbb C$, $v=|\psi'|$ belongs to BHO by the Koebe distortion theorem \cite{Pom92}*{Cor. 1.1.5}; see e.g. \cite{GW24}*{Prop. 2.2}.

We now prove that \eqref{e:Bp-nec-sym} 
is necessary for $\Pi^u\colon L^p(\mathbb{D},u^2)\to L^{p,\infty}(\mathbb{D},u^2)$ whenever $u \in \text{BHO}$. Theorem \ref{thm:Nec} then follows as an immediate application.
\begin{theorem}\label{prop:GeneralNecessity}
If $1 \le p < \infty$, $u \in \BHO$, and $\Pi^u$ is bounded from $L^{p}(\mathbb{D},u^2)$ to $L^{p,\infty}(\mathbb{D},u^2)$, then $u$ satisfies \eqref{e:Bp-nec-sym}.
\end{theorem}
\begin{proof}
Applying the lower bounds from Lemma \ref{l:avg-below}, we obtain the necessary conditions
	\begin{equation} \norm{u^{-1}}_{L^{p,\infty}(Q,u^2)} \lesssim \norm{u^{-1}}_{L^{p}(T_Q^k,u^2)}, \label{eq:Necess1}  \end{equation}
	\begin{equation} \frac{\norm{u^{-1}}_{L^{p,\infty}(T_Q^k,u^2)} \norm{u^{-1}}_{L^{p'}(Q,v^2)}}{|Q|}\lesssim 1,\label{eq:Necess2}
   \end{equation}
whenever $\ell(Q) \le \frac{1}{20}$. Indeed, \eqref{eq:Necess1} comes from testing on $u^{-1} \mathbf 1_{T_Q^k}$, and \eqref{eq:Necess2} from Lemma \ref{l:single-scale}. It remains to show
	\[ \norm{u^{-1}}_{L^{p}(T_Q^k,u^2)} \lesssim \norm{u^{-1}}_{L^{p,\infty}(T_Q^k,u^2)}.\] 
But this is an immediate consequence of the BHO property, 
since $T_Q^k$ is contained in $T_P$ for some (larger) cube $P$.

To handle the case $\ell(Q) \ge \frac{1}{20}$, we simply note that both quantities
	\[ \norm{u^{-1}}_{L^{p,\infty}(\mathbb D, u^2)}, \quad \norm{u^{-1}}_{L^{p'}(\mathbb D,u^2)} \]
must be finite. Indeed, $\mathbb D$ can be written as a finite union of cubes $P$ with $\ell(P) \le \frac 1{10}$ together with a finite union of top-halves $T_{P'}$. But then, for $\ell(Q) \ge \frac 1{20}$, 
	\[ \norm{u^{-1}}_{L^{p,\infty}(Q,u^2)}\norm{u^{-1}}_{L^{p'}(Q,u^2)} \le \norm{u^{-1}}_{L^{p,\infty}(\mathbb D, u^2)} \norm{u^{-1}}_{L^{p'}(\mathbb D,u^2)} \lesssim |Q|.\]
This concludes the proof.
\end{proof}

\begin{proof}[Proof of Theorem \ref{thm:Nec}]
The result follows from Theorem \ref{prop:GeneralNecessity} by the change of variables in Proposition \ref{p:change}
and the fact that $|\psi'| \in \BHO$. 
\end{proof}

\section{Reduction to the Bergman maximal operator}\label{s:Reduction}

In this section, we will reduce our sought-after estimates for $\Piv$ to certain ones for $\Mv$. We first prove the following Coifman-Fefferman inequality, relying on the pointwise bound \eqref{e:sparse} which, as discussed there, also holds for $\Pi^+$.

\begin{proposition}\label{p:CF}
If $0<p<\infty$ and $u \in B_\infty$, then 
\[
    \int_{\mathbb{D}} |\Pi^+ f|^p u \, dA \lesssim \int_{\mathbb{D}} |Mf|^p u \, dA
\]
for all $f$ for which the right-hand side is finite.
\end{proposition}

\begin{proof}[Proof of Proposition \ref{p:CF}]
Let $f$ be a nonnegative and bounded function with compact support. First, in the case $0<p\leq 1$, we have by H\"older's inequality that
	\[ \int_{\mathbb D} \Big|\sum_{Q \in \mathcal D} A_Q f\Big|^p u \, dA \le \sum_Q \avg{f}_Q^p u(Q). \]
Using the inequality $u(Q) \lesssim u(T_Q)$ by virtue of $u \in B_\infty$, we have
	\[ \avg{f}_Q^p u(Q) \lesssim \avg{f}_Q^p u(T_Q) \le \int_{T_Q} \abs{Mf}^p u \, dA.\]
Summing the above display over $Q$ achieves the result. 
For the case $p>1$, we use duality. For $g \in L^{p'}(\mathbb{D},u)$, again appealing to the $B_{\infty}$ condition, we see
\[
\Big|\int_{\mathbb D} A_Q fg u \, dA \Big| \lesssim \avg{f}_Q \avg{g}_{u,Q} u(T_Q) \le \int_{T_Q} Mf M_ug u\, dA,
\]
where 
	\[ M_u g := \sup_{Q \in \mathcal D} \avg{|g|}_{u,Q} \mathbf 1_Q. \] 
It is well-known that the dyadic structure entails that $M_u$ is bounded on $L^p(\mathbb{D},u)$ independently of $u$ for each $1< p < \infty$ \cite{LN19}*{Theorem 15.1}. Summing over $Q \in \mathcal D$ concludes the proof.\looseness=-1
\end{proof}

We will also use the following reverse H\"older property of $\BHO$ weights satisfying the $B_{p}$ condition.
\begin{lemma}\label{l:RH}
Given $K \ge 1$ and $1 \le p < \infty$, there exist $C,\tau > 1$ such that 
 	\[ \avg{u}_{\tau,Q} \le C \avg{u}_Q \]
for all $u \in B_p\cap\BHO$ with 
 	$\max\{ [u]_{\BHO}, [u]_{B_p} \} \le K$
and all Carleson boxes $Q$.
\end{lemma}
\noindent We refer to any number of works \cites{APR19,SW23,Bor04} for a proof of this property. Furthermore, there has been recent interest in finding suitable weaker conditions than $\BHO$ to guarantee the reverse H\"older property \cites{MP25,Gok25}. For those more familiar with $A_p$ classes, the usual proof of the reverse H\"older property \cite{CF74} goes through unchanged, simply noting that the $\BHO$ condition implies $u \lesssim Mu$ (which may fail for general $B_p$ weights). We note that Lemma \ref{l:RH} remains true when $p=\infty$ in a slightly more delicate form that will not be needed here; see \cite{SW23}. 

\noindent We use Proposition \ref{p:CF} and Lemma \ref{l:RH} to reduce the mixed-weighted weak-type estimate for $\Pi$ to the analogous bound for $M$ in the $p=1$ endpoint case. 

\begin{theorem}\label{thm:Pi-M1}
If $u \in B_1$, then 
	\[ \norm{\Pi^u\colon L^1(\mathbb{D},u^2) \to L^{1,\infty}(\mathbb{D},u^2) } \lesssim \norm{M^u\colon L^1(\mathbb{D},u^2) \to L^{1,\infty}(\mathbb{D},u^2) }. \]
\end{theorem}
\begin{proof}
Our proof is inspired by the approach of \cite{CUMP05} in the two-weight, Calder\'on-Zygmund setting. Given $u \in B_1$, we get that $M^u$ is bounded on $L^\infty(\mathbb{D})$ since for $z \in Q$
    \[ u^{-1}(z) \avg{uf}_Q \le \norm{f}_{L^{\infty}(\mathbb{D})} [u]_{B_1}.\]
Hence, by interpolation with the $L^2(\mathbb{D},u^2)$ bound \cite{Gra14}, we obtain that $M^u$ is bounded on the Lorentz space $L^{s,1}(\mathbb{D},u^2)$ for $2<s<\infty$. Furthermore, these operator norms are uniformly bounded by, say, $K_0$ if we restrict our attention to $s$ larger than a fixed exponent, say $s>4$ \cite{CUMP05}.

 We will implement the Rubio de Francia algorithm with the maximal operator $M^u$ and certain $s$ to be chosen so that by duality, with $\frac 1r + \frac 1s =1$,
    \begin{equation}\label{e:duality} \begin{aligned} \norm{ (\Pi^+)^u f}_{L^{1,\infty}(\mathbb{D},u^2)}^{\frac 1r} &= \norm{ \left[ (\Pi^+)^u f\right]^{\frac 1r} }_{L^{r,\infty}(\mathbb{D},u^2)} \\
    &= \sup_{\substack{ h \in L^{s,1}(\mathbb{D},u^2) \\ \norm{h}=1}} \int_{\mathbb{D}} \left|\Pi^+ (uf) \right|^{\frac 1r} h u^{2- \frac 1r} \, dA. \end{aligned} \end{equation}
With the Rubio de Francia operator $R := \sum_{k=0}^\infty (2K_0)^{-k} (M^u)^k$ we have $h \le Rh$, $Rh \cdot u$ belongs to $B_1$ with characteristic $2K_0$, and $R$ is bounded on $L^{s,1}(\mathbb{D},u^2)$ with norm $2$ for any $s>4$. Therefore, factoring $Rh \cdot u^{2-\frac 1r} = (Rh \cdot u)u^{\frac 1s}$, we need to choose $s$ large enough, depending only on $K_0$, so that this weight is indeed in $B_\infty$ (with $C$ and $\delta$ from \eqref{e:Binfty} depending only on $K_0$). In fact, at the end of the proof, we will find $C,s>4$ depending only on $K_0$ such that $[(Rh \cdot u)u^{\frac 1s}]_{B_1} \le C$, which suffices by the discussion following \eqref{e:Binfty}.

Assuming such $C$ and $s$ exist for now, in light of \eqref{e:duality}, we fix $h$ in $L^{s,1}(\mathbb{D},u^2)$ of unit norm and estimate
    \[ \begin{aligned} \int_{\mathbb{D}} \abs{\Pi^+ (uf)}^{\frac 1r} \abs{h} u^{2- \frac 1r} \, dA &\le \int_{\mathbb{D}} \abs{\Pi^+ (uf)}^{\frac 1r} (R\abs{h} \cdot u) u^{\frac 1s} \, dA \\
    &\lesssim \int_{\mathbb{D}} \abs{M (uf)}^{\frac 1r} (R\abs{h} \cdot u)u^{\frac 1s} \, dA \end{aligned} \]
using the Coifman-Fefferman inequality (Proposition \ref{p:CF}) with $p= \frac 1r$ and weight $(R\abs{h} \cdot u) u^{\frac 1s}$, which we chose to belong to $B_\infty$. Now, by H\"older's inequality and the $L^{s,1}(\mathbb{D},u^2)$ boundedness of $R$, we obtain
    \[\begin{aligned} \int_{\mathbb{D}} \abs{M(u f)}^{\frac 1r} (R\abs{h} \cdot u) u^{\frac 1s} \, dA &= \int \abs{M^u f}^{\frac 1r} R\abs{h} \cdot u^2 \\
        &\le \norm{\left(M^u f\right)^{\frac 1r}}_{L^{r,\infty}(\mathbb{D},u^2)} \norm{R \abs h}_{L^{s,1}(\mathbb{D},u^2)} \\
        &\lesssim \norm{M^u f}_{L^{1,\infty}(\mathbb{D},u^2)}^{\frac 1r}.  \end{aligned} \]

It only remains to show that $(Rh \cdot u)u^{\frac 1s}$ is indeed in $B_\infty$ for $s$ large enough, depending only on $K_0$. First, we show that $Rh \cdot u$ belongs to $\BHO$. 
To see this, notice that $Rh \cdot u$ is an infinite sum of terms of the form $Mh_k$. Fix a Carleson box $Q$, $z,w$ in $T_Q$, and a Carleson box $P$ containing $z$. By dilating $P$ by a factor of at most $3$, we find $\tilde P$ which contains all of $T_Q$. In particular, $\tilde P$ contains both $P$ and $w$, thus
    \[ \avg{ h_k }_{P} \le 9 \avg{h_k}_{\tilde P} \le  9 Mh_k(w).\]
But then taking the supremum over all such $P$ shows $Mh_k(z) \lesssim Mh_k(w)$. Since this estimate holds uniformly over $h_k$ and over and $z,w$ in $T_Q$, we have shown that	
	\[ [ (Rh) \cdot u] _{\BHO} \le 9.\] 
Together with the fact that $[Rh \cdot u]_{B_1} \le 2K_0$, we may invoke Lemma \ref{l:RH} to obtain $C,\tau>1$ depending only on $K_0$ such that for all cubes $Q$,
    \[ \avg{Rh \cdot u}_{\tau,Q} \le C \avg{Rh \cdot u}_Q.\]
Now finally we choose $s$ to be the H\"older conjugate of $\tau$ so that
    \[ \avg{(Rh \cdot u) u^{\frac 1s} }_{Q} \le \avg{Rh \cdot u }_{\tau, Q}\avg{ u^{\frac 1s} }_{s, Q} \le C \avg{ Rh \cdot u}_Q \avg{u}_Q^{\frac 1 {s}}.\]
However, since both $Rh \cdot u$ and $u$ are in $B_1$, we conclude that $(Rh \cdot u)u^{\frac 1s}$ also belongs to $B_1$.
\end{proof}

\begin{remark}\label{r:unified}
Notice that in the above proof, our weight condition $u \in B_1$ implied the conjugated maximal function $M^u$ was bounded on $L^\infty(\mathbb D,u^2) = L^\infty(\mathbb D)$. If we interpret statement C. in Theorem \ref{thm:summ} as $L^{\infty,1} = L^\infty$ then there is no distinction in the statement of the Theorem when $p >1$ versus $p=1$. Indeed, by testing on the constant function $\mathbf 1$, we see
    \[ \norm{ M^u \colon L^\infty \to L^\infty} = [u]_{B_1}.\]
So in fact, in this case, C. and D. from Theorem \ref{thm:summ} are equivalent.

However, such an interpretation should be cautioned against. The dual space of $L^{1,\infty}$ is not $L^\infty$ so there is no reason to believe $M^u$ to be bounded on $L^\infty(u^2)$ if it was bounded from $L^1(u^2)$ to $L^{1,\infty}(u^2)$. Moreover, the same logic applies to $\Pi^u$, yet more strongly since $\Pi$ itself is not even bounded on $L^\infty(\mathbb D)$. From this analysis (together with the example from Section \ref{Subsection:Intro-Necessity}), it is evident that $M^v\colon L^\infty(v^2) \to L^\infty(v^2)$ is not necessary for the sought-after weak-type $(1,1)$ bound for $\Pi_\Omega$.
\end{remark}

\subsection{Reflexive case}
From our perspective, there is no point in proving Theorem \ref{thm:Pi-M1} in the case $p>1$ since we already know from \cite{Bek86} that $v^{2-p} \in B_p$ is sufficient for the mixed weighted weak-type bound for $\Piv$. 
Instead, we obtain the following sufficient condition which is a priori weaker than $v^{2-p} \in B_p$, as stated in Theorem \ref{thm:summ}. To prove this, we will use the fact that the necessary condition \eqref{e:Bp-nec-sym}
implies that the dual weight $u^{2-p'}$ belongs to $B_\infty$.

\begin{lemma}\label{l:dual}
    Let $1 < p < \infty$. If $u$ satisfies \eqref{e:Bp-nec-sym} 
    and $q > p'$, then $u^{2-p'} \in B_q$. 
\end{lemma}
\begin{proof}
The statement is trivial for $p=2$, so we may assume $p \neq 2$. Let $\sigma = u^{2-p'}$. $v= u^{\frac{1}{2-p'}}.$ Notice that for any $p<s<2$ (if $p<2$) and $2<s<p$ (if $p>2$), by interpolation
    \[ \norm{u^{-1}}_{L^s(Q,u^2)} \lesssim_s \norm{u^{-1}}_{L^{p,\infty}(Q,u^2)}^\theta \norm{u^{-1}}_{L^2(Q,u^2)}^{1-\theta}, \quad \frac{1}{s} = \frac{\theta}{p} + \frac{1-\theta}{2}.\]
Set now
    \[ q= \frac{p'-2}{2-s}+1.\] 
Therefore, $u^{2-s} = \sigma^\frac{-1}{q-1}$. Thus we have
    \[ \avg{ \sigma^{-1}}_{\frac{1}{q-1},Q}\avg{\sigma}_Q 
     \lesssim \left( \frac{\norm{u^{-1}}_{L^{p,\infty}(Q,u^2)}}{|Q|^{\frac 1p}} \right)^{s\theta(q-1)}  \frac{\norm{u^{-1}}_{L^{p'}(Q,u^2)}^{p'}}{|Q|}.\]
However, some algebra shows that $s \theta(q-1) = p'$, so the condition \eqref{e:Bp-nec-sym} yields $\sigma \in B_q$. 
Finally, as $s$ ranges over $(p, 2)$ or $(2,p)$, $q$ ranges over $(p',\infty)$.
\end{proof}

\begin{theorem}\label{thm:Bp-disk}
If $u$ is a weight, $1<p<\infty$, and $M^u$ is bounded from $L^{p',1}(\mathbb{D}, u^2)$ to $L^{p'}(\mathbb{D},u^2)$, then $(\Pi^+)^u$ is bounded from $L^p(\mathbb{D},u^2) \to L^{p,\infty}(\mathbb{D},u^2)$.
\end{theorem}
Now, Theorem \ref{thm:Bp} is an immediate consequence of Proposition \ref{p:CF}. 

\begin{proof}[Proof of Theorem \ref{thm:Bp}]
First, since $\Pi^+$ is self-adjoint on $L^2(\mathbb D)$, $(\Pi^+)^u$ is self-adjoint on $L^2(\mathbb{D},u^2)$, and so we have equality of operator norms 

    \[ \norm{(\Pi^+)^u\colon L^{p}(u^2) \to L^{p,\infty}(u^2) } = \norm{(\Pi^+)^u\colon L^{p',1}(u^2) \to L^{p'}(u^2) } \] 
by duality. Furthermore, for any $f \in L^{p'}(u^2)$,
	\[ \int_{\mathbb{D}} |(\Pi^+)^u f|^{p'} u^2 \, dA = \int_{\mathbb{D}} |\Pi^+(uf)|^{p'} \sigma \, dA, \quad \sigma =u^{2-p'}.\]
Now we invoke the assumed boundedness of $\Mv$ twice. First, by Proposition \ref{prop:nec}, this implies the necessary condition on the weight \eqref{e:Bp-nec-sym}, which forces $\sigma \in B_\infty$ by Lemma \ref{l:dual}. Therefore, we can apply Proposition \ref{p:CF} to control the above display by
	\[ \int_{\mathbb{D}} |M(uf)|^{p'} \sigma \, dA = \norm{M^u(f)}_{L^{p'}(u^2)}^{p'} \lesssim \norm{f}_{L^{p',1}(\mathbb{D},u^2)}^{p'}.\]
\end{proof}

\section{Bergman maximal function estimates}\label{s:max}
We prove some estimates for the conjugated Bergman maximal function $\Mv$ that complement Theorem \ref{thm:Bp-disk} and Theorem \ref{thm:Pi-M1}.
First, we study the case $p=1$ to establish Theorem \ref{thm:B1}. The strategy follows Muckenhoupt and Wheeden's original work from \cite{MW77} in the Calder\'on-Zygmund setting; see also \cite{CUMP05} for the two-weight generalization.
\begin{theorem}\label{thm:max-1} If $u$ belongs to $B_1 \cap \BHO$, then $M^u$ is bounded from $L^1(\mathbb{D}, u^2)$ to $L^{1,\infty}(\mathbb{D},u^2)$.
\end{theorem}
When combined with Theorem \ref{thm:Pi-M1} connecting $\Pi^u$ and $M^u$, we obtain Theorem \ref{thm:sawyer} from the introduction. Then, applying the change of variables in Proposition \ref{p:change} we obtain our first main result, Theorem \ref{thm:B1}. In the proof of Theorem \ref{thm:max-1}, we'll use the following Bergman space analogue of \cite{MW77}*{Lemma 1}, which is a direct consequence of the reverse H\"older inequality of Lemma \ref{l:RH}. 
\begin{lemma}\label{lem:b1}
    If $u \in B_1\cap \BHO$, then there exists $\delta>0$ such that 
    \[
        |\{ z \in Q\colon u(z) > \lambda \}| \lesssim |Q|\left( \frac{\avg{u}_Q}{\lambda}\right)^{1 + \delta}
    \]
    for any Carleson cube $Q$ and $\lambda > 0$.
\end{lemma}
\begin{proof}
This follows immediately from Chebyshev's inequality and Lemma \ref{l:RH}. Indeed, for every $\lambda>0$ and $\tau>1$,
	\[ \left( \frac{| \{ z \in Q \colon u(z) > \lambda\} |\lambda^\tau}{|Q|} \right)^{\frac{1}{\tau}}  \le \avg{u}_{\tau,Q} \lesssim \avg{u}_Q.\]
Taking $1+\delta=\tau$ proves the claim.
\end{proof}

\begin{proof}[Proof of Theorem \ref{thm:max-1}]
    Let $u \in B_1$ and $f \in L^1(\mathbb{D},u^2)$. Recall the desired estimate
    \[
        \sup_{\lambda > 0} u^2(\{z \in \mathbb{D} \colon M^uf(z) > \lambda \}) \lesssim \frac{1}{\lambda}\|f\|_{L^1(\mathbb{D},u^2)}
    \]
    By homogeneity and the change of variables $f \mapsto uf$, it suffices to prove
    \[
        u^2(\{z \in \mathbb{D} \colon Mf(z) > u(z)\}) \lesssim \|f\|_{L^1(\mathbb{D},u)}.
    \]
    Moreover, by homogeneity, we may also assume that $u(\mathbb{D}) = 1$, and, by standard limiting arguments, that $f \in L^\infty(\mathbb{D})$ and $f \ge 0$. Let $a > 4$ and $q \in \mathbb{N}$ satisfying $a^{q-1} \le [u]_{B_1} < a^q$ (recall that $[u]_{B_1} \ge 1$). Therefore, as $u \in B_1$,
    \[
        1 = \int_\mathbb{D}u\,dA \le [u]_{B_1}\operatorname{ess\,inf}_{z \in \mathbb{D}} u(z) < a^q\operatorname{ess\,inf}_{z \in \mathbb{D}}u(z).
    \]
    It follows that $u(z) \ge a^{-q}$ a.e. on $\mathbb{D}$. Recall that for $f \in L^\infty(\mathbb{D})$ and any $k \in \mathbb Z$,
    \[
        \{z \in \mathbb{D} \colon Mf(z) > a^k \} = \bigcup_{j}Q_{k,j}
    \]
    for some collection of maximal dyadic Carleson boxes $\{Q_{k,j}\}_j$ for which $\avg{f}_{Q_{j,k}} \sim a^k$. Therefore, 
    \begin{align}
        \notag u^2(\{z \in \mathbb{D} \colon Mf(z) > u(z) \}) 
        &= \sum_{k = -q}^\infty u^2\left(\{ a^k < u \le a^{k+1}\} \cap \{ Mf > u \}\right)\\
        \notag &\le \sum_{k = -q}^\infty u^2\left(\{ a^k < u \le a^{k+1}\} \cap \{ Mf > a^k \}\right)\\
        &= \sum_{k = -q}^\infty \sum_{j}u^2(\{Q_{k,j} \colon a^k < u \le a^{k+1}\}). \label{e:B1-step-1}
    \end{align}
   These cubes are nested: for $-q \le m \le k$, each $Q_{k,j}$ is contained in a unique $Q_{m,i}$. We take the sum above over those cubes $Q_{k,j}$ for which $\{Q_{k,j} \colon a^k < u \le a^{k+1}\} \neq \emptyset$. For any such cube $Q_{k,j}$, we have
    \begin{equation}\label{e:avg_est}
        \avg{u}_{Q_{j,k}} \le [u]_{B_1} a^{k+1}
    \end{equation}
    We will sum the measures of these sets based on the $v$-averages on their associated cubes. We say a cube $Q_{k,j}$ is principal if $k = -q$, or if for the smallest principal cube containing $Q_{k,j}$, say $Q_{m,i}$, we have
    \[
        \avg{u}_{Q_{k,j}} \ge 2 \avg{u}_{Q_{m,i}}.
    \]
    Let $P := \{(k,j) \colon Q_{k,j} \text{ is principal}\}$ and, for $(k,j) \in P$ and $m \ge k$, let $F(k,j,m)$ denote the collection of indices $(m,i)$ for which $Q_{k,j}$ is the smallest principal cube containing $Q_{m,i}$. By the aforementioned nestedness, each cube is contained in some principal cube. Therefore, applying Lemma~\ref{lem:b1}, if $(m,i) \in F(k,j,m)$, then
    \[ u^2(\{Q_{m,i} \colon a^m < u \le a^{m+1}\})\lesssim a^{2m}|\{Q_{m,i} \colon a^m < u \le a^{m+1}\}| \] 
    \[ \lesssim a^{(1-\delta)m}|Q_{m,i}|\avg{u}_{Q_{m,i}}^{1+\delta} \lesssim a^{(1-\delta)m}|Q_{m,i}|\avg{u}_{Q_{k,j}}^{1+\delta} .
    \]
    By maximality of $Q_{k,j}$, we have that $Mf(z) = M(f\chi_{Q_{k,j}})(z)$ for each $z \in Q_{k,j}$. Consequently, by the weak-type $(1,1)$ estimate for the maximal operator, we have 
    \[
        \sum_{(m,i) \in F(k,j,m)}|Q_{m,i}| \le |\{ Q_{k,j} \colon M(f\chi_{Q_{k,j}}) > a^m \}| \lesssim a^{-m}\int_{Q_{k,j}}f\,dA.
    \]
   Therefore, combining the above two displays,
    with \eqref{e:avg_est}, we have
    \begin{align*}
        \eqref{e:B1-step-1}
        &\lesssim \sum_{(k,j) \in P}\sum_{m=k}^\infty \sum_{(m,i) \in F(k,j,m)} a^{m(1-\delta)}\avg{u}_{Q_{j,k}}^{1 + \delta}|Q_{m,i}| \\
        &\lesssim \sum_{(k,j) \in P} \avg{u}_{Q_{j,k}}^{1 + \delta}\left(\int_{Q_{k,j}}f\,dA\right)\sum_{m=k}^\infty a^{-m\delta}\\
        &\lesssim \sum_{(k,j) \in P} \avg{u}_{Q_{j,k}} \left(\int_{Q_{k,j}}f\,dA\right)\\   
        &= \int_{\mathbb{D}}f(z)\left(\sum_{(k,j) \in P}\avg{u}_{Q_{k,j}}\chi_{Q_{k,j}}(z)\right)\,dA(z).
    \end{align*}
    Since dyadic Carleson boxes have finite overlap, each $z \in \mathbb{D}$ belongs to at most finitely many such boxes, hence to at most finitely many principal cubes. For each $z \in \mathbb{D}$, let $Q_k(z)$ denote the principal cube at level $k$ containing $z$ (when one exists) and let $Q(z) = Q_{k_0}(z)$ denote the smallest such principal cube. Then
    \begin{align*}
        \sum_{(k,j) \in P}{|Q_{k,j}|}\avg{u}_{Q_{k,j}}\chi_{Q_{k,j}}(z) 
        &= \sum_{k = -q}^{k_0}\avg{u}_{Q_k(z)}\\
        &\le \avg{u}_{Q(z)}\sum_{k = -q}^{k_0}2^{-k}\\
        &\lesssim 2^q\avg{u}_{Q(z)} \\
        &\lesssim [u]_{B_1}^2u(z).
    \end{align*}
    Applying this estimate, we conclude
    \[
        u^2(\{z \in \mathbb{D} \colon Mf(z) > u(z) \}) \lesssim \|f\|_{L^1(\mathbb{D},u)},
    \]
    as desired.
\end{proof}

We next turn to the case $p>1$. We are unable at this time to characterize the weights for the dual maximal function estimate $M^u \colon L^{p',1}(\mathbb{D},u^2) \to L^{p'}(\mathbb{D},u^2)$, providing the sought-after companion to Theorem \ref{thm:Bp-disk}. 
However, we can adapt the strategy of \cite{Swe25} to characterize the weights for which $M^u\colon L^p(\mathbb{D},u^2) \to L^{p,\infty}(\mathbb{D},u^2)$ as precisely those that satisfy the necessary condition \eqref{e:Bp-nec-sym}.

\begin{theorem}\label{thm:max-p} If $1<p<\infty$ and $u$ satisfies \eqref{e:Bp-nec-sym}, then $M^u$ is bounded from $L^p(\mathbb{D},u^2)$ to $L^{p,\infty}(\mathbb{D},u^2)$.
\end{theorem}

\begin{proof}
By replacing $f$ by $fu^{-1}$, it suffices to prove
    \[ \norm{ u^{-1} Mf}_{L^{p,\infty}(\mathbb{D},u^2)} \le \norm{f}_{L^p(\mathbb{D},u^{2-p})}. \]
To this end, for each $k\in \mathbb Z$, introduce the maximal dyadic cubes $\{Q_j^k\}_{j \in J_k}$ such that
    \begin{equation}\label{e:maximality} \avg{f}_{Q_j^k} > 2^k.\end{equation}
By maximality and doubling, $\avg{f}_{Q_j^k} \le 4 \cdot 2^k$ so each $Q_{j}^k$ can appear at most $2$ times in the collection $\mathcal Q:=\{ Q_j^k\colon k \in \mathbb Z, j \in J_k\}$. In particular, this means $\{T_Q \colon Q \in \mathcal Q\}$ has finite overlap (of at most $2$). Therefore, with $\dot Q_{j}^k := Q_{j}^k \setminus \{Mf > 2^{k+1}\}$, 
    \[ \lambda^p u^2 \left( \left\{ z \in \mathbb D \colon u^{-1}(z) Mf(z) > \lambda \right\} \right) \le \sum_{j,k}  \lambda^p u^2\left( \left\{ z \in \dot Q_{j^k} \colon u^{-1}(z) > 2^{-(k+1)} \lambda \right\} \right). \]
Furthermore, the condition \eqref{e:Bp-nec-sym} and \eqref{e:maximality} gives for each $Q=Q_j^k$ with $\sigma=u^{2-p'}$  
    \[\lambda^p u^2\left(\left\{ z \in Q \colon u^{-1}(z) > 2^{-(k+1)} \lambda \right\} \right) \lesssim 2^{kp} \sigma(Q)^{1-p} |Q|^p \]
    \[ \lesssim \sigma(Q)^{1-p} \left( \int_{Q}|f| \, dA \right)^p = \sigma(Q) \avg{f\sigma^{-1}}_{\sigma,Q}^p.\]
Now, using Lemma \ref{l:dual} so that $\sigma(Q) \lesssim \sigma(T_Q)$ and the finite overlap of $T_Q$, we obtain
    \[ \lambda u^2 \left( \left\{  z \colon u^{-1}(z) Mf(z) > \lambda \right\} \right)^{\frac 1p} \lesssim \left( \sum_{Q} \sigma(T_Q) \avg{f\sigma^{-1}}_{\sigma,Q}^p \right)^{\frac 1p} \]
    \[ \lesssim \norm{M_{\sigma}(f\sigma^{-1})}_{L^p(\sigma)} \lesssim \norm{ f\sigma^{-1}}_{L^p(\mathbb{D},\sigma)} = \norm{f}_{L^p(\mathbb{D},u^{2-p})},\]
as desired.
\end{proof}

\section{Remarks and Consequences} \label{s:RC}
In this section, we prove Corollary \ref{c:A1}, provide an illustrative example, and record some additional consequences of Theorem \ref{thm:B1}. In what follows, we slightly abuse notation and write $\Pi_{\Omega}$ and $\Pi_{\Omega}^{+}$ for the unique $L^1(\Omega)$ to $L^{1,\infty}(\Omega)$ bounded extensions of $\Pi_{\Omega}$ and $\Pi_{\Omega}^{+}$, respectively. The proof of the following technical Lemma is included in the Appendix.

\begin{lemma}\label{cor:AbsInt}
Let $\Omega \subset \mathbb{C}$ be any simply connected domain for which $\Pi_{\Omega}^+$ extends to a bounded operator $L^1(\Omega) \rightarrow L^{1,\infty}(\Omega)$. Then the following hold:
\begin{enumerate}
\item For every $f \in L^1(\Omega)$ and $z \in \Omega$,
\begin{equation*}
    \int_{\Omega}|K_{\Omega}(z,w)||f(w)|dA(w)<\infty \label{AbsInt}
\end{equation*}
\item If $f \in L^1(\Omega)$, there holds
\begin{equation*} \Pi_{\Omega}f \equiv \int_{\Omega}K_{\Omega}(\cdot,w)f(w)\,dA(w) \label{BergmanKernelL1Formula} \end{equation*}
where the symbol $\equiv$ is interpreted as an equality of measurable functions (defined a.e.) in $L^{1,\infty}(\Omega)$. Moreover, the function on the right-hand side is holomorphic on $\Omega$.
\item If $\psi\colon \mathbb{D} \rightarrow \Omega$ is a conformal map onto $\Omega$, then $\psi'$ is bounded from below on $\mathbb{D}$.
\end{enumerate}
\end{lemma}
\begin{corollary} \label{cor:B1DomainsAbsInt}
Any simply connected domain $\Omega$ with $\psi: \mathbb{D} \rightarrow \Omega $ conformal satisfying $|\psi'| \in B_1$ admits the conclusions of Lemma \ref{cor:AbsInt}.
\end{corollary}

We are now in a position to prove Corollary \ref{c:A1}.
\begin{proof}[Proof of Corollary \ref{c:A1}]
By Proposition \ref{p:density} in the Appendix, we have $A^2(\Omega) \cap A^1(\Omega) $ is dense in $A^1(\Omega)$.
Let $f \in A^1(\Omega)$ and take functions $f_n \in A^2(\Omega) \cap A^1(\Omega)$ converging to $f$ in the $L^1(\Omega)$ norm. We have $\Pi_\Omega f_n=f_n$ for all $n$ in light of the fact $f_n \in A^2(\Omega)$. Therefore, by Theorem \ref{thm:B1}, we have 
$$
\|\Pi_{\Omega} f- f_n\|_{L^{1,\infty}(\Omega)} = \|\Pi_{\Omega}(f-f_n) \|_{L^{1,\infty}(\Omega)} \lesssim \|f-f_n\|_{L^1(\Omega)} \rightarrow 0
$$
as $n \rightarrow \infty$. By convergence in $L^{1,\infty}(\Omega)$ and passing to a subsequence, we may assume $f_n$ converges to $\Pi_{\Omega}f$ almost everywhere. Passing to a further subsequence, we have $f_n \rightarrow f$ almost everywhere, which gives $\Pi_{\Omega}f=f$ as elements of $A^1(\Omega)$. The integral expression formula follows from Lemma \ref{cor:AbsInt}. Finally, $\Pi_\Omega$ produces $A^{1,\infty}$ functions from $L^1$ data due to Theorem \ref{thm:B1} combined with Lemma \ref{cor:AbsInt}.\looseness=-1 
\end{proof}

The following example demonstrates that the Bergman projection can exhibit extremely irregular behavior at the endpoint $p=1$ despite being bounded on $L^p$ in the full reflexive range.

\begin{example}
There exists a simply connected, bounded domain $\Omega \subset \mathbb{C}$ satisfying the following properties:
\begin{enumerate}
    \item There is a holomorphic function $f \in A^1(\Omega)$ such that 
    $$ \int_{\Omega} |K_{\Omega}(z,w)||f(w)|\, dA(w)=+\infty$$ for any $z \in \Omega$;
    \item $\Pi_{\Omega}$ does not extend to be weak-type $(1,1)$;
    \item $\Pi_{\Omega}$, $\Pi_{\Omega}^{+}$ are bounded on $L^p(\Omega)$ for $1<p<\infty$.
\end{enumerate}

\begin{proof}
We make use of Example 2.2.1 from \cite{GW24}. In particular, we set $\Omega= \psi(\mathbb{D})$, where
$$ \psi(z)= \frac{( \frac{z+1}{4})}{\log( \frac{z+1}{4})}.$$ It was already shown in \cite{GW24} that $|\psi'|^2 \notin B_1$; the same argument gives $|\psi'| \notin B_1$ as $|\psi'|$ is not bounded from below.
Item (3) was already proved in \cite{GW24} (technically, it was only stated for $\Pi_{\Omega}$, but the same argument works for $\Pi_{\Omega}^{+}$).

To prove item (2), for all $n \in \mathbb{N}$ with $n \geq 2$, let $\varepsilon(n)>0$ be small enough so that $D_n:=D(-\frac{1}{n}, \varepsilon(n)) \subset \Omega$. Set $f_n= \dfrac{\mathbf{1}_{D_n}}{|D_n|}$ and note by construction $\|f_n\|_{L^1(\Omega)}=1$ for all $n$. To disprove the weak-type bound, it is enough to show $\|\Pi_\Omega f_n\|_{L^{1,\infty}(\Omega)}$ blows up as $n \rightarrow \infty$. This goal will be accomplished if we show there is a measurable set $E \subset \Omega$, $|E|>0$ such that for all $z \in E$, $\lim_{n \rightarrow \infty} |\Pi_\Omega f_n(z)|=+\infty.$ Let $z_0 \in \Omega$ be arbitrary, and let $B_{z_0,r} \subset \Omega$ be a small ball of radius $r>0$ centered at $z_0$. For any $z \in B_{z_0,r}$ we calculate, using the mean value property for anti-holomorphic functions and the change of variable formula:
\begin{align*}
\Pi_{\Omega}f_n(z) & = \frac{1}{|D_n|}\int_{D_n}K_{\Omega}(z,w) \, dA(w) = K_{\Omega}(z,-1/n) \\  
&= \frac{1}{\pi \psi'(\psi^{-1}(z))} \left(\frac{1}{(1-\overline{\psi^{-1}(-1/n)} \psi^{-1}(z))^2}\right) \frac{1}{ \psi'(\psi^{-1}(-1/n))}.\end{align*}
Let $p_n= \psi^{-1}(-1/n)$. It is clear $p_n \rightarrow -1$ as $n \rightarrow \infty$, the above formula shows there is a constant $C_{z_0,r}>0$ such that 
 $|\Pi_{\Omega}f_n(z)| \geq C_{z_0,r} \dfrac{1}{|\psi'(p_n)|},$
and the right hand side blows up as $n \rightarrow \infty$ since $|\psi'(z)| \sim \dfrac{1}{\left \lvert \log(\frac{z+1}{4} )\right \rvert} $ on $\mathbb{D}$, proving the claim. 

We now turn to the proof of item (1). A computation shows that it suffices to exhibit a function $f \in A^1(\Omega) $ (equivalently, $f$ is holomorphic on $\Omega$ and $(f\circ \psi)|\psi'|^2 \in L^1(\mathbb{D})$) which satisfies $(f \circ \psi) |\psi'| \notin L^1(\mathbb{D})$. Indeed, if $\Pi_{\Omega}^+|f|$ were absolutely integrable, a change of variables with $\zeta=\psi(\xi)$ gives
\begin{align*}
\int_{\Omega}|K_{\Omega}(z,\zeta)||f(\zeta)|dA(\zeta) & = \frac{1}{\pi |\psi'(\psi^{-1}(z))|} \int_{\mathbb{D}}\dfrac{|f \circ \psi(\xi)|}{|1-\overline{\xi}\,\psi^{-1}(z)|^2} |\psi'(\xi)| \, dA(w) \\ 
& \geq C_z \int_{\mathbb{D}} |f \circ \psi| |\psi'| \, dA,
\end{align*}
where $$ C_z:=\frac{1}{\pi |\psi'(\psi^{-1}(z))|} \inf_{\xi \in \mathbb{D}} \dfrac{1}{|1-\overline{\xi}\,\psi^{-1}(z)|^2}>0. $$ Let $f(z)=\frac{1}{(\psi^{-1}(z)+1)^2}$. Clearly, $f$ is holomorphic on $\Omega$ and satisfies $f \circ \psi(z)= \frac{1}{(z+1)^2}$ on $\mathbb{D}$. A computation shows that $\varphi(z):= \frac{z+1}{4}$ conformally maps $\mathbb{D}$ onto $D(\frac{1}{4},\frac{1}{4}) \subset \frac{1}{2}\mathbb{D}$. A straightforward change of variable calculation gives: 
\begin{align*}
  \int_{\mathbb{D}} |f \circ \psi(z)||\psi'(z)|^2 \, dA(z) 
 & \lesssim \int_{\frac{1}{2} \mathbb{D}}  \dfrac{1}{|\xi|^2|\log{|\xi|}|^2} \\
 & = 2 \pi \int_{0}^{1/2} \dfrac{-1}{r (\log{r})^2} \, dr< \infty,
\end{align*}
so $f \in A^1(\Omega)$.

On the other hand, straightforward calculations show that for sufficiently small $\varepsilon, \delta>0$, the sector
$$A_\varepsilon^\delta:= \{r e^{\mathrm{i} \theta}: 0<r<\varepsilon, \, |\theta|<\delta\}$$
is fully contained in $D(\frac{1}{4}, \frac{1}{4}).$
In addition, we can assume $\delta \leq |\log{\varepsilon}|$, so that $|\log{|\xi|}| \gtrsim |\log{\xi}|$ for all $\xi \in A_\varepsilon^\delta$. We then readily see by integrating in polar coordinates
$$ \int_{A_\varepsilon^\delta} \dfrac{1}{|\xi|^2|\log{|\xi|}|}=+\infty,$$
which implies 
$$(f \circ \psi) |\psi'| \notin L^1(\mathbb{D}), $$ as desired.
\end{proof}
\end{example}

We next show that $A^{1,\infty}(\Omega)$ is a closed subspace of $L^{1,\infty}(\Omega)$. This result does not rely on the Bergman projection at all, only the fact that convergence in the $A^{1,\infty}(\Omega)$ topology is stronger than uniform convergence on compact sets.
\begin{proposition} \label{ClosedQuasiBanach}
 If $\Omega \subsetneq \mathbb C$ is any domain, then $A^{1,\infty}(\Omega)$ is a closed subspace of $L^{1,\infty}(\Omega)$. 
 In particular, $A^{1,\infty}(\Omega)$ is a quasi-Banach space of analytic functions. 
\begin{proof}
Take $\{f_n\} \subset A^{1,\infty}(\Omega)$ converging to $f \in L^{1,\infty}(\Omega)$ in the $L^{1,\infty}(\Omega)$ topology. To conclude that $f$ is analytic, it is enough to show that the sequence $\{f_n\}$ is uniformly Cauchy on compact subsets of $\Omega$, for then the $f_n$ converge uniformly on compact subsets to an analytic function $\widetilde{f}$, and we must have $\widetilde{f}=f$. 
Fix a compact set $K\subset \Omega$ and $z \in K$. Let $B_z$ be a disk centered at $z$ with radius $\frac{1}{2} \operatorname{dist}(z, \partial \Omega) \geq \frac{1}{2} \operatorname{dist}(K, \partial \Omega)>0$, so that $B_z \subset \Omega$. For each $n,m \in \mathbb{N}$, the function $|f_n-f_m|^{1/2}$ is subharmonic on $\Omega$, so the sub-mean value property implies 
$$ |f_n(z)-f_m(z)| \leq \left( \frac{1}{|B_z|}\int_{B_z} |f_n-f_m|^{1/2} \, dA \right)^2 \lesssim \|f_n-f_m\|_{L^{1/2}(\Omega)}.$$
Applying \cite{Gra14}*{Exercise 1.1.11}, the last display is dominated by a constant times $\|f_n-f_m\|_{L^{1,\infty}(\Omega)}$, which goes to $0$ by assumption. The result follows. 
\end{proof}
\end{proposition}

\begin{corollary} \label{c:PiOmegaClosure}
Let $\Omega \subset \mathbb C$ be a simply connected domain bounded by a Jordan curve and $\psi\colon \mathbb D \to \Omega$ be a conformal map. If $|\psi'| \in B_1$, then $$ \overline{\Pi_\Omega L^1(\Omega)}^{L^{1,\infty}} \subset A^{1,\infty}(\Omega).$$
\begin{proof}
Take $f_n \in L^2(\Omega) \cap L^1(\Omega)$ converging to $f$ in the $L^1(\Omega)$ norm, so that $\lim_{n \rightarrow \infty} \|\Pi_\Omega f_n-\Pi_\Omega f\|_{L^{1,\infty}(\Omega)}=0$ by Theorem \ref{thm:B1}. But $\Pi_\Omega f_n \in A^{1,\infty}(\Omega)$ for all $n$, so the result follows from Proposition \ref{ClosedQuasiBanach}. 
\end{proof}
\end{corollary}

\begin{remark}
    By Corollaries \ref{cor:B1DomainsAbsInt} and \ref{c:PiOmegaClosure}, we have that 
    $$
        A^1(\Omega) \subset \Pi_{\Omega} L^1(\Omega) \subset \overline{\Pi_\Omega L^1(\Omega)}^{L^{1,\infty}} \subset A^{1,\infty}(\Omega)
    $$
    whenever $\Omega\subset \mathbb{C}$ is a simply connected domain 
    with conformal map $\psi\colon \mathbb{D}\rightarrow\Omega$ satisfying $|\psi'| \in B_1$. In particular, $\Pi_{\Omega}$ reproduces $A^1(\Omega)$ functions and produces $A^{1,\infty}(\Omega)$ functions from $L^1(\Omega)$ data. 
\end{remark}

Below, we write $L \log L(\Omega)$ to denote the space of measurable functions $f$ on $\Omega$ equipped with finite Luxemburg norm

$$ \|f\|_{L \log L(\Omega)}:= \inf\left\{\lambda>0: \int_{\Omega} \Phi\left(\dfrac{|f|}{\lambda}\right) \, dA \leq 1 \right \}< \infty, \quad \Phi(t)=t \log^{+}t.$$ We derive the following Kolmogorov and Zygmund-type inequalities directly from Theorem \ref{thm:B1}. The proofs are standard and omitted (see \cite{SW22}*{Section 4}).
\begin{corollary}\label{cor:Kolmogorov}
Let $\Omega \subset \mathbb C$ be a simply connected domain and $\psi\colon \mathbb D \to \Omega$ be a conformal map. If $K \subseteq \Omega$ has finite area, $0<p<1$, and $|\psi'|\in B_1$, then 
$$ \|\Pi_{\Omega}f \|_{L^p(K)} \lesssim \|f\|_{L^{1}(\Omega)}$$
for all $f \in L^1(\Omega)$. 
\end{corollary}

\begin{corollary}\label{cor:Zygmund}
Let $\Omega \subset \mathbb C$ be a bounded, simply connected domain and $\psi\colon \mathbb D \to \Omega$ be a conformal map. If $|\psi'|\in B_1$, then 
$$ \|\Pi_{\Omega}f \|_{L^1(\Omega)} \lesssim \|f\|_{L \log L(\Omega)}$$
for all $f \in L^1(\Omega)$. 
\end{corollary}

\section{Appendix}\label{s:appendix}
In this appendix, we collect a few functional analysis results which were used in the applications of Section \ref{s:RC}.
\begin{proposition}\label{p:density}
    If $1 \le p < \infty$ and $\psi: \mathbb D \to \mathbb C$ is a conformal map satisfying $|\psi'|^{2-p} \in B^p$, then $A^2(\Omega) \cap A^p(\Omega)$ is dense in $A^p(\Omega)$, where $\Omega = \psi(\mathbb D)$. 
\end{proposition}
We follow the method of Hedenmalm \cite{Hed02}, which relies on a sufficient condition for cyclic vectors in the Bergman space of the disk $A^p(\mathbb D)$ due to Brown, Korenblum, and Shapiro \cites{BK88,Sha67}.

Let us introduce the notation $\mathcal P$ for the vector space of all polynomials in $z$. Then, given $0<p<\infty$ and $f \in A^p(\mathbb D)$ define the invariant subspace
    \[ I(f,p) = \overline{\{ Pf : P \in \mathcal P \}}^{A^p(\mathbb D)}. \]
We say $f$ is cyclic in $A^p$ if $I(f,p) = A^p(\mathbb D)$.
\begin{theorem}[\cite{HKZ00}*{Theorem 3.7}]\label{thm:cyclic}
Let $0<r<p<q<\infty$ and $f \in A^q$. If $f$ is cyclic in $A^r$ then $f$ is cyclic in $A^p$.
\end{theorem}
    
\begin{lemma}
    $I(f,p) = A^p(\mathbb D)$ if and only if $1 \in I(f,p)$.
\end{lemma}
\begin{proof}
   Assume $1 \in I(f,p).$ Let $g \in A^p(\mathbb D)$ and $\ep>0$. We can find a polynomial $P$ such that
        \[ \norm{P-g}_{A^p(\mathbb D)} < \frac \ep 2.\]
    On the other hand, we can find a polynomial $Q$ such that 
        \[ \norm{Qf-1}_{A^p(\mathbb D)} < \frac{\ep}{2 \norm{P}_\infty}.\]
    Then, $\norm{PQf-g}_{A^p(\mathbb D)} < \ep$ by the triangle inequality.
\end{proof}

\begin{proof}[Proof of Proposition \ref{p:density}]
    Set $f=(\psi')^{\frac{2}{p}-1}$. 
    
    \textit{Step 1: There exists $r>0$ such that $f$ is cyclic in $A^r$.}
The $B_p$ condition implies that $f,1/f \in A^s(\mathbb D)$ for some $s>0$. Therefore, we can find a sequence of polynomials $P_n$ converging to $1/f$ in $A^s(\mathbb D)$. But then,
    \[ \int_{\mathbb D} |P_n f -1 |^{\frac s2}\, dA = \int_{\mathbb D} |P_n - 1/f|^{\frac s2} |f|^{\frac s2} \, dA \le \left( \norm{P_n-f^{-1}}_{A^s} \norm{f}_{A^s}\right)^{\frac 12} \to 0\]
as $n \to \infty$.

    \textit{Step 2: $f \in A^q$ for some $q>p$.} Since $f^p = |\psi'|^{2-p} \in B_p$, by the reverse H\"older property, Lemma \ref{l:RH}, there exists $\tau>1$ such that $f^{\tau p}$ is integrable, i.e. $f \in A^q$ for $q=\tau p > p$.
    
    Applying Theorem \ref{thm:cyclic}, Steps 1 and 2 now show that $f$ is cyclic in $A^p$. 
    
    \textit{Step 3: Polynomials are dense in $A^p(u)$ with $u=|\psi'|^{2-p} = |f|^p$.} Indeed, given $g \in A^p(u)$, we see that $gf \in A^p$. Since $f$ is cyclic in $A^p$, we can find $P_nf \to gf$ in $A^p$, which translates to $P_n \to g$ in $A^p(u)$.
    
    \textit{Step 4: $A^2(\Omega) \cap A^p(\Omega)$ is dense in $A^p(\Omega)$.} Recall from \eqref{e:transformation} the composition operator $C_\psi h = (h \circ \psi ) \psi'$, $h:\Omega \to \mathbb C$ which satisfies
        \[ \norm{h}_{A^p(\Omega)} = \norm{C_\psi h}_{A^p(\mathbb D,u)}\]
    for every $0<p<\infty$. So, given $h \in A^p(\Omega)$ set $g=C_\psi h$. By step 3, we can find a sequence of polynomials $P_n \to g$ in $A^p(u)$. Then, with $Q_n = C_\psi^{-1}P_n \in A^p(\Omega)$ we have $Q_n \to h$ in $A_p(\Omega)$. But 
        \[ \int_{\Omega} |Q_n|^2 = \int_{\mathbb D} |P_n|^2,\]
    so $Q_n \in A^2(\Omega) \cap A^p(\Omega)$. 
\end{proof}

 \begin{proof}[Proof of Lemma \ref{cor:AbsInt}]

A routine approximation argument with the Monotone Convergence Theorem gives $G(z):= \int_{\Omega}|K_{\Omega}(z,w)||f(w)|dA(w)$ is an extended real-valued measurable function on $\Omega$, so it remains to show that $G(z)$ is finite everywhere. We first prove $G \in L^{1,\infty}(\Omega)$, hence finite almost everywhere. This follows from the weak-type estimate because a limiting argument gives for every $\lambda>0$
$$ \lambda |\{z \in \Omega\colon G(z)>\lambda\}| \leq \|\Pi_{\Omega}^{+}\|_{L^1 \rightarrow L^{1,\infty}} \|f\|_{L^1(\Omega)}<\infty.  $$

To upgrade almost everywhere finiteness to everywhere finite, observe that for such a point $z \in \Omega$, by change of variable with $\zeta=\psi(\xi)$, 
\begin{align*}
\int_{\Omega}|K_{\Omega}(z,\zeta)||f(\zeta)|dA(\zeta) & = \frac{1}{\pi |\psi'(\psi^{-1}(z))|} \int_{\mathbb{D}}\dfrac{|f \circ \psi(\xi)|}{|1-\overline{\xi}\,\psi^{-1}(z)|^2} |\psi'(\xi)| \, dA(w) \\ 
& \geq C_z \int_{\mathbb{D}} |f \circ \psi| |\psi'| \, dA,
\end{align*}
where $$ C_z:=\frac{1}{\pi |\psi'(\psi^{-1}(z))|} \inf_{\xi \in \mathbb{D}} \dfrac{1}{|1-\overline{\xi}\,\psi^{-1}(z)|^2}>0. $$
This proves $(f \circ \psi) \psi' \in L^1(\mathbb{D})$, and then the Dominated Convergence Theorem, continuity of $\psi'$ and $K_{\mathbb{D}}$, and change of variable give that the function $G(z)$ is continuous in the variable $z$, hence finite everywhere.

To prove \eqref{BergmanKernelL1Formula}, reduce matters to proving the identity for positive $f$, take a monotone approximating sequence, and apply the Dominated Convergence Theorem. To see the integral expression is holomorphic in the variable $z$, pass $\overline \partial$ inside the integral using the Dominated Convergence Theorem and the fact that $K_\Omega(z,w)$ is holomorphic in the variable $z$.

Finally, to prove item (3), it is enough to show that the multiplication operator $M_{|\psi'|^{-1}}$ is bounded on $L^1(\mathbb{D}).$ Let $h \in L^1(\mathbb{D})$, fix $z \in \Omega$ and let $K \subset \Omega$ be a fixed compact set containing $z$. We see $f:=\frac{h \circ \psi^{-1}}{|\psi'\circ \psi^{-1}|^2} \in L^1(\Omega)$ with $\|f\|_{L^1(\Omega)}=\|h\|_{L^1(\mathbb{D})}$, so change of variable implies 

$$\int_{\mathbb{D}}\left \lvert \frac{h}{\psi'}  \right \rvert\, dA \leq C_K \,\Pi_{\Omega}^{+}(|f|)(z), \quad z \in K.$$
Taking weak-type norms of both sides, we get
\begin{align*}
 |K| \int_{\mathbb{D}} | M_{|\psi'|^{-1}}h | \, dA & \leq C_K \|\Pi_{\Omega}^{+}(|f|)\|_{L^{1,\infty}(K)} \\
 & \leq C_K \, \|\Pi_{\Omega}\|_{L^1 \rightarrow L^{1,\infty}} \|f\|_{L^1(\Omega)}= C_K \, \|\Pi_{\Omega}\|_{L^1 \rightarrow L^{1,\infty}} \|h\|_{L^1(\mathbb{D})},
 \end{align*}
 which proves the claim. 
\end{proof}

\bibliographystyle{alpha}
\bibliography{main}

\end{document}